\documentclass{amsart}
\usepackage{amsmath,amssymb,bm}
\usepackage{amsmath,amsthm,amssymb,cases}
\usepackage[all]{xy}
\usepackage{amsmath}
\usepackage{amssymb}
\usepackage{amscd}
\usepackage{amsthm}
\usepackage[dvips]{graphicx}
\usepackage{color}

\theoremstyle{definition}
\newtheorem{defn}{Definition}[section]

\theoremstyle{plain}
\newtheorem{thm}[defn]{Theorem}
\newtheorem{prop}[defn]{Proposition}
\newtheorem{lem}[defn]{Lemma}

\newtheorem{claim}[defn]{Claim}
\numberwithin{equation}{section}
\title[A finite presentation for $\operatorname{Aut}(H_1(N_g;\mathbb Z_2),\cdot )$]{A finite presentation for the automorphism group of the first homology of a non-orientable surface over $\mathbb Z_2$ preserving the mod $2$ intersection form}
\author[R.~Kobayashi]{Ryoma Kobayashi}
\address{
(Ryoma Kobayashi)
Department of Mathematics,
Faculty of Science and Technology,
Tokyo University of Science,
Noda, Chiba, 278-8510, Japan
}
\email{kobayashi\_ryoma@ishikawa-nct.ac.jp}
\author[G.~Omori]{Genki Omori}
\address{
(Genki Omori)
Department of Mathematics,
Tokyo Institute of Technology,
Oh-okayama, Meguro, Tokyo 152-8551, Japan
}
\email{omori.g.aa@m.titech.ac.jp}
\date{\today}
\begin{document}
\maketitle
\begin{abstract}
Let $\operatorname{Aut}(H_1(N_g;\mathbb Z_2),\cdot )$ be the group of automorphisms on the first homology group with $\mathbb Z_2$ coefficients of a closed non-orientable surface $N_g$ preserving the mod $2$ intersection form. In this paper, we obtain a finite presentation for $\operatorname{Aut}(H_1(N_g;\mathbb Z_2),\cdot )$. As an application we calculate the second homology group of $\operatorname{Aut}(H_1(N_g;\mathbb Z_2),\cdot )$. 
\end{abstract}

\section{Introduction}

For $g\geq 1$ and $n\geq 0$, let $N_{g,n}$ be a compact connected
non-orientable surface of genus $g$ with $n$ boundary components (we denote $N_{g,0}$ by $N_g$) and a bilinear form $\cdot :H_1(N_g;\mathbb
Z_2)\times H_1(N_g;\mathbb Z_2) \rightarrow \mathbb Z_2$
the mod $2$ intersection form on the first homology group
$H_1(N_g;\mathbb Z_2)$ of $N_g$ with $\mathbb Z_2$ coefficients. We
represent $N_g$ by a sphere with $g$ crosscaps as in Figure~\ref{basis_h1}, i.e. we regard $N_g$ as a sphere with $g$ boundary components
attached a M\"{o}bius band to each boundary component.
We define $\operatorname{Aut}(H_1(N_g;\mathbb Z_2),\cdot )$ by the
subgroup of the automorphism group
$\operatorname{Aut}H_1(N_g;\mathbb Z_2)$ of $H_1(N_g;\mathbb Z_2)$
preserving the mod $2$ intersection form $\cdot $ . Note that $\operatorname{Aut}(H_1(N_g;\mathbb Z_2),\cdot )$ is isomorphic to $O(g,\mathbb Z_2)=\{A\in GL(g,\mathbb Z_2)\ |\ ^t\!AA=E\}$ by taking the basis $\{x_1, x_2, \dots , x_g \}$ for $H_1(N_g;\mathbb Z_2)$, where $x_i$ is a homology class of a one-sided simple closed curve $\mu _i$ in Figure~\ref{basis_h1} and $E$ is an identity matrix of  $GL(g,\mathbb Z_2)$ (cf. \cite{Mccarthy-Pinkall}). By Korkmaz \cite{Korkmaz} and Szepietowski \cite{Szepietowski1} we have isomorphisms
\[
\operatorname{Aut}(H_1(N_g;\mathbb Z_2),\cdot ) \cong \left\{ \begin{array}{ll}
 \operatorname{Sp}(2h,\mathbb Z_2)&\text{if} \ g=2h+1,   \\
 \operatorname{Sp}(2h,\mathbb Z_2)\ltimes \mathbb Z_2^{2h+1}&\text{if} \ g=2h+2.
 \end{array} \right.
\]

Let $a_i\ (i=1,\dots ,g-1,\ \text{for}\ g\geq 2), \ b\ (\text{for}\ g\geq 4)\in \operatorname{Aut}(H_1(N_g;\mathbb Z_2),\cdot )$ be the following elements:
\[
a_i :\left\{ \begin{array}{llll}
 x_i & \mapsto & x_{i+1}, & \\
 x_{i+1} & \mapsto & x_i, & \\
 x_k & \mapsto & x_k & (k\not=i,i+1 ), \\
 \end{array} \right.
b :\left\{ \begin{array}{lll}
 x_1 & \mapsto & x_2+x_3+x_4,   \\
 x_2 & \mapsto & x_1+x_3+x_4,   \\
 x_3 & \mapsto & x_1+x_2+x_4,   \\
 x_4 & \mapsto & x_1+x_2+x_3,   \\
 x_k & \mapsto & x_k  \hspace{1cm}(k\not=1,2,3,4 ).
 \end{array} \right. 
\]

\begin{figure}[h]
\includegraphics[scale=0.80]{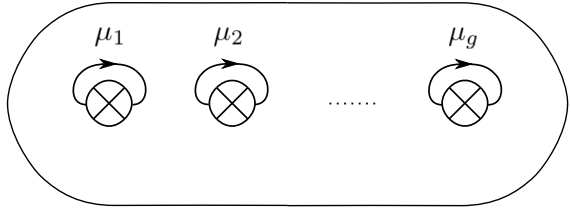}
\caption{Simple closed curves $\mu _1, \mu _2, \dots , \mu _g$ in $N_g$ representing the basis $x_1, x_2, \dots , x_g$ for $H_1(N_g;\mathbb Z_2)$ respectively.}\label{basis_h1}
\end{figure}

In this paper, we give a finite presentation for $\operatorname{Aut}(H_1(N_g;\mathbb Z_2),\cdot )$.

\begin{thm}\label{thm-main}

If $g=1$, $2$, $3$, then $\operatorname{Aut}(H_1(N_g;\mathbb Z_2),\cdot )$ is the following group.

\begin{itemize}
 \item $\operatorname{Aut}(H_1(N_1;\mathbb Z_2),\cdot )=1$,
 \item $\operatorname{Aut}(H_1(N_2;\mathbb Z_2),\cdot )=\bigl< a_1 \big| a_1^2=1\bigr>\cong \mathbb Z_2$,
 \item $\operatorname{Aut}(H_1(N_3;\mathbb Z_2),\cdot )=\bigl< a_1, a_2 \big| a_1^2=a_2^2=(a_1a_2)^3=1\bigr>$.
\end{itemize}
If $g\geq 4$ is odd, $g=4$ or $6$, then $\operatorname{Aut}(H_1(N_g;\mathbb Z_2),\cdot )$ admits a presentation with generators $a_1, \dots , a_{g-1}, b$ and relations:

\begin{enumerate}
 \item $a_i^2=b^2=1$\hspace{2cm} for $i=1,\dots ,g-1$,
 \item $(a_ia_j)^2=1$\hspace{2cm} for $g\geq 4,\ |i-j|>1$,
 \item $(a_ia_{i+1})^3=1$\hspace{2cm} for $g\geq 3,\ i=1,\dots ,g-2$,
 \item $(a_ib)^2=1$\hspace{2cm} for $g\geq 4,\ i\not=4$,
 \item $(a_4b)^3=1$\hspace{2cm} for $g\geq 5$,
 \item $(a_2a_3a_4a_5a_6b)^{12}=(a_1a_2a_3a_4a_5a_6b)^9$\hspace{2cm} for $g\geq 7$.
\end{enumerate}
If $g\geq 8$ is even, then $\operatorname{Aut}(H_1(N_g;\mathbb
 Z_2),\cdot )$ admits a presentation with generators $a_1, \dots
 , a_{g-1}, b, b_0, \dots , b_{\frac{g-2}{2}}$ and relations
 {\rm (1)-(6)} above and the following relations:

\begin{enumerate}
 \item[(7)] $b_0=a_1,\ b_1=b,\ b_2=(a_1a_2a_3a_4a_5b)^5$, 
 \item[(8)] $b_{i+1}=(b_{i-1}a_{2i}a_{2i+1}a_{2i+2}a_{2i+3}b_i)^5(b_{i-1}a_{2i}a_{2i+1}a_{2i+2}a_{2i+3})^{-6}$\hspace{2cm} for $2\leq i\leq \frac{g-4}{2}$,
 \item[(9)] $[a_{g-5},b_{\frac{g-2}{2}}]=1$.
\end{enumerate}
\end{thm}

We read every word of every group in this paper from right to left. In Section~\ref{proof-thm}, we will prove Theorem~\ref{thm-main} for $g\geq 4$. 
Theorem~\ref{thm-main} is clear for $g=1,2$. For $g=3,4$, Szepietowski \cite[in the proof of Theorem 5.5]{Szepietowski2} gave the presentation for $\operatorname{Aut}(H_1(N_g;\mathbb Z_2),\cdot )$. Note that $\operatorname{Aut}(H_1(N_3;\mathbb Z_2),\cdot )$ is isomorphic to the dihedral group $D_6$ and the symmetric group $S_3$. 
By the result of Korkmaz \cite[Corollary 4.1]{Korkmaz} and Theorem~\ref{thm-main}, the first homology group of $\operatorname{Aut}(H_1(N_g;\mathbb Z_2),\cdot )$ is as follows.
\[
\hspace{0.5cm}
H_1(\operatorname{Aut}(H_1(N_g;\mathbb Z_2),\cdot );\mathbb Z)= \left\{ \begin{array}{lll}
\hspace{0.3cm} 0 & \hspace{0.8cm} \text{for} \ g=1,\ g\geq 7, \\
 \bigl< [a_1]\bigr> \cong \mathbb Z_2 & \hspace{0.8cm} \text{for} \ g=2,\ 3,\ 5,\ 6,\\
 \bigl< [a_1],[b]\bigr> \cong \mathbb Z_2\oplus \mathbb Z_2 & \hspace{0.8cm} \text{for} \ g=4. 
 \end{array} \right. 
\]
Note that the above equality is known for $g\geq 7$ odd (see, for instance, \cite{Taylor}).

In Section~\ref{H_1andH_2}, by using the presentation for $\operatorname{Aut}(H_1(N_g;\mathbb Z_2),\cdot )$ obtained in Theorem~\ref{thm-main}, we calculate the second homology group of $\operatorname{Aut}(H_1(N_g;\mathbb Z_2),\cdot )$ for $g\geq 9$. We get the following theorem.

\begin{thm}\label{h_2}

For $g\geq 9$ or $g=7$, the second homology group of $\operatorname{Aut}(H_1(N_g;\mathbb Z_2),\cdot
 )$ is trivial.

\end{thm}

Theorem~\ref{h_2} was shown by Stein \cite{Stein} for odd $g$ (see Theorem 2.13 and Proposition 3.3 (a)). More precisely, Stein proved $H_2(\operatorname{Sp}(2h,\mathbb Z_m);\mathbb Z)=0$ when $h\geq 3$ and $m$ is not divisible by $4$(see also \cite{Funar-Pitsch}).

To prove Theorem~\ref{h_2}, we give a generating set for $H_2(\operatorname{Aut}(H_1(N_g;\mathbb Z_2),\cdot );\mathbb Z)$ which consists of one element $x_0$ by an application of the discussion of Pitsch \cite{Pitsch}. By using the generator of $H_2(\operatorname{Aut}(H_1(N_g;\mathbb Z_2),\cdot );\mathbb Z)$ and Stein's result we show that $x_0$ is trivial for $g\geq 9$. 

\section{Preliminaries}

Let $\alpha _1, \dots , \alpha _{g-1}, \beta $ be two-sided
simple closed curves on $N_g$ as in Figure~\ref{scc}. Arrows on the side of simple closed curves in Figure~\ref{scc} indicate directions of Dehn twists along their simple closed curves. Since the actions of the Dehn twists along $\alpha _1,\dots , \alpha_{g-1}, \beta $ induce $a_1,\dots , a_{g-1}, b$ on $H_1(N_g;\mathbb Z_2)$ respectively, we denote Dehn twists along $\alpha _1,\dots , \alpha _{g-1}, \beta $ by $a_1,\dots , a_{g-1}, b$ and abuse the notation.

\begin{figure}[h]
\includegraphics[scale=0.90]{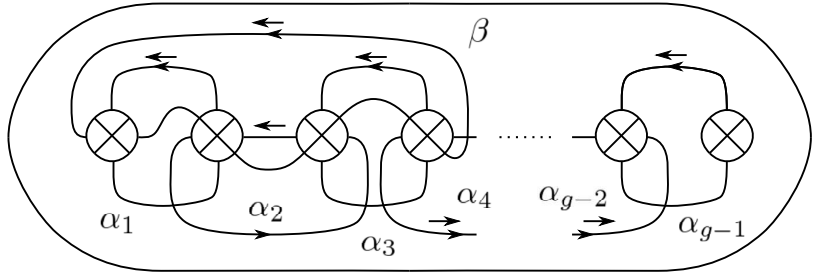}
\caption{Simple closed curves $\alpha _1,\dots , \alpha _{g-1}, \beta $ on $N_g$.}\label{scc}
\end{figure}

Let $\mu $ be a one-sided simple closed curve and $\alpha $ a two-sided simple closed curve such that $\mu $ and $\alpha $ intersect transversely in one point. For these simple closed curves $\mu$ and $\alpha $, we denote by $Y_{\mu , \alpha }$ a self-diffeomorphism on $N_g$ which is described as the result of pushing the regular neighborhood of $\mu $ once along $\alpha $ (see Figure~\ref{yhomeo}). We call $Y_{\mu , \alpha }$ a {\it Y-homeomorphism}. We set the direction of $Y_{\mu _i, \alpha _j} \ (1\leq i\leq g,\ 1\leq j\leq g-1)$ by the orientation of $\alpha _j$ in Figure~\ref{scc} and $y:=Y_{\mu _1, \alpha _1}$. Note that the action of Y-homeomorphism on $H_1(N_g;\mathbb Z_2)$ is trivial. 

\begin{figure}[h]
\includegraphics[scale=0.60]{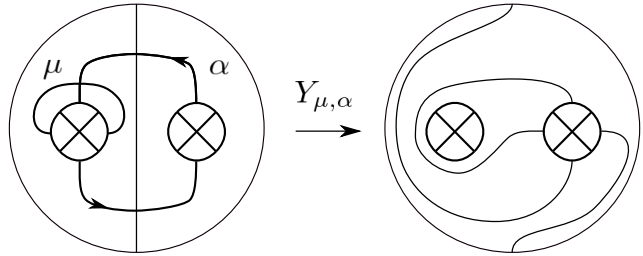}
\caption{The Y-homeomorphism on the regular neighborhood of $\mu \cup \alpha $.}\label{yhomeo}
\end{figure}

The {\it mapping class group} $\mathcal{M}(N_{g,n})$ of $N_{g,n}$ is the group of isotopy classes of self-diffeomorphisms on $N_g$ fixing each boundary component pointwise. Paris and Szepietowski \cite{Paris-Szepietowski} gave a finite presentation for $\mathcal{M}(N_g)$. The presentation has a generating set which consists of Dehn twists along two-sided simple closed curves and ``crosscap transpositions''.  Stukow \cite{Stukow1} obtained a finite presentation for $\mathcal{M}(N_g)$ whose generators are Dehn twists and a Y-homeomorphism. Stukow's presentation is the following.

\begin{thm}[\cite{Stukow1}]\label{Stukow}
If $g\geq 4$ is odd or $g=4$, then $\mathcal{M}(N_g)$ admits a presentation with generators $a_1,\dots , a_{g-1},b, y$ and $\rho $. The defining relations are
\begin{enumerate}
 \item[(A1)] $[a_i,a_j]=1$\hspace{1cm} for $|i-j|>1$,
 \item[(A2)] $a_ia_{i+1}a_i=a_{i+1}a_ia_{i+1}$\hspace{1cm} for $i=1,\dots ,g-2$,
 \item[(A3)] $[a_i,b]=1$\hspace{1cm} for $i\not=4$,
 \item[(A4)] $a_4ba_4=ba_4b$\hspace{1cm} for $g\geq 5$,
 \item[(A5)] $(a_2a_3a_4b)^{10}=(a_1a_2a_3a_4b)^6$\hspace{1cm} for $g\geq 5$,
 \item[(A6)] $(a_2a_3a_4a_5a_6b)^{12}=(a_1a_2a_3a_4a_5a_6b)^9$\hspace{1cm} for $g\geq 7$,
 \item[(B1)] $y(a_2a_3a_1a_2ya_2^{-1} a_1^{-1}a_3^{-1}a_2^{-1}) = (a_2a_3a_1a_2ya_2^{-1}a_1^{-1}a_3^{-1}a_2^{-1})y$,
 \item[(B2)] $y(a_2a_1y^{-1}a_2^{-1}ya_1a_2)y=a_1(a_2a_1y^{-1}a_2^{-1}ya_1a_2)a_1$,
 \item[(B3)] $[a_i,y]=1$\hspace{1cm} for $i=3,\dots ,g-1$,
 \item[(B4)] $a_2(ya_2y^{-1}) = (ya_2y^{-1})a_2$,
 \item[(B5)] $ya_1=a_1^{-1}y$,
 \item[(B6)] $byby^{-1} = \{a_1a_2a_3(y^{-1}a_2y)a_3^{-1}a_2^{-1}a_1^{-1} \}\{a_2^{-1}a_3^{-1}(ya_2y^{-1})a_3a_2\}$,
 \item[(B7)] $[(a_4a_5a_3a_4a_2a_3a_1a_2ya_2^{-1}a_1^{-1}a_3^{-1}a_2^{-1}a_4^{-1}a_3^{-1}a_5^{-1}a_4^{-1}),b] =1$\hspace{1.0cm} for $g\geq 6$,
 \item[(B8)] $\{(ya_1^{-1}a_2^{-1}a_3^{-1}a_4^{-1})b(a_4a_3a_2a_1y^{-1})\}\{(a_1^{-1}a_2^{-1}a_3^{-1}a_4^{-1})b^{-1}(a_4a_3a_2a_1)\}$
	     $=\{(a_4^{-1}a_3^{-1}a_2^{-1})y(a_2a_3a_4)\}\{a_3^{-1}a_2^{-1}y^{-1}a_2a_3\}\{a_2^{-1}ya_2\}y^{-1}$\hspace{0.9cm} for $g\geq 5$,
 \item[(C1a)] $(a_1a_2\cdots a_{g-1})^g=\rho $\hspace{1.0cm} for $g$ odd,
 \item[(C1b)] $(a_1a_2\cdots a_{g-1})^g=1$\hspace{1.0cm} for $g$ even,
 \item[(C2)] $[a_1,\rho ]=1$,
 \item[(C3)] $\rho ^2=1$,
 \item[(C4a)] $(y^{-1}a_2a_3\cdots a_{g-1}ya_2a_3\cdots a_{g-1})^{\frac{g-1}{2}}=1$\hspace{1.0cm} for $g$ odd,
 \item[(C4b)] $(y^{-1}a_2a_3\cdots a_{g-1}ya_2a_3\cdots a_{g-1})^{\frac{g-2}{2}}y^{-1}a_2a_3\cdots a_{g-1}=\rho $\hspace{1.0cm} for $g$ even,
\end{enumerate}
where $[X,Y]=XYX^{-1}Y^{-1}$.
If $g\geq 6$ is even then $\mathcal{M}(N_g)$ admits a presentation with generators $a_1,\dots , a_{g-1}, y, b,\rho $ and $b_0,\dots , b_{\frac{g-2}{2}}$. The defining relations are {\rm (A1)-(A6), (B1)-(B8), (C1a)-(C4b)} above and the following relations:
\begin{enumerate}
 \item[(A7)] $b_0=a_1,\ \ b_1=b$,
 \item[(A8)] $b_{i+1}=(b_{i-1}a_{2i}a_{2i+1}a_{2i+2}a_{2i+3}b_i)^5(b_{i-1}a_{2i}a_{2i+1}a_{2i+2}a_{2i+3})^{-6}$\\
	     for $1\leq i\leq \frac{g-4}{2}$,
 \item[(A9a)] $[b_2,b]=1$\hspace{1.0cm} for $g=6$,
 \item[(A9b)] $[a_{g-5},b_{\frac{g-2}{2}}]=1$\hspace{1.0cm} for $g\geq 8$.
\end{enumerate}
\end{thm}

Relations (A1) and (A3) are called {\it disjointness relations} and relations (A2) and (A4) are called {\it braid relations}. When we deform relations (or words) by disjointness relations and braid relations, we write ``DI'' and ``BR'' on the left-right arrow (or the equality sign) respectively.

\begin{figure}[h]
\includegraphics[scale=0.80]{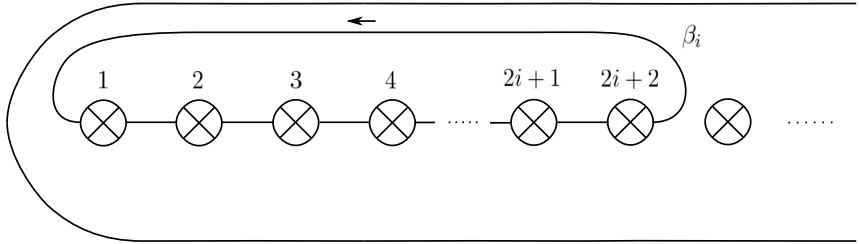}
\caption{Simple closed curves $\beta _i$ on $N_g$ for $2\leq i\leq \frac{g-2}{2}$.}\label{beta_i}
\end{figure}

\begin{figure}[h]
\includegraphics[scale=0.75]{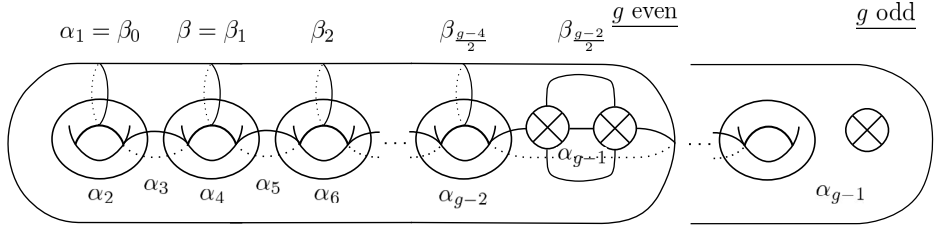}
\caption{A different view of simple closed curves $\alpha _i\ (1\leq i\leq g-1)$ and $\beta _j\ (0\leq i\leq \frac{g-2}{2})$ on $N_g$.}\label{different}
\end{figure}

$b_i\ (2\leq i\leq \frac{g-2}{2})$ in the $g$ even case of Theorem~\ref{Stukow} is the Dehn twist along a simple closed curve $\beta _i$ in Figure~\ref{beta_i}. The arrow on the side of the simple closed curve $\beta _i$ in Figure~\ref{beta_i} indicates the direction of the Dehn twist $b_i$. We note that $N_g$ is diffeomorphic to a surface as in Figure~\ref{different} and we can choose the diffeomorphism such that simple closed curves $\alpha _i\ (1\leq i\leq g-1)$ in Figure~\ref{scc} and $\beta _j\ (0\leq j\leq \frac{g-2}{2})$ in Figure~\ref{beta_i} are sent to a position in Figure~\ref{different}.

Since the action of $\mathcal{M}(N_g)$ on $H_1(N_g;\mathbb Z_2)$ preserves the mod $2$ intersection form $\cdot $, we have a homomorphism $\rho _2:\mathcal{M}(N_g)\rightarrow \operatorname{Aut}(H_1(N_g;\mathbb Z_2),\cdot )$. McCarthy and Pinkall \cite{Mccarthy-Pinkall} showed that $\rho _2$ is surjective. $\Gamma _2(N_g):=\operatorname{ker}\rho _2$ is called the {\it level $2$ mapping class group} of $\mathcal{M}(N_g)$. Szepietowski \cite{Szepietowski2} proved that $\Gamma _2(N_g)$ is generated by Y-homeomorphisms for $g\geq 2$. More precisely, Szepietowski showed the following theorem.

\begin{thm}[\cite{Szepietowski2}]\label{Szepietowski}

For $g\geq 2$, $\Gamma _2(N_g)$ is normally generated by $y$ in $\mathcal{M}(N_g)$.

\end{thm}  

We note that squares of Dehn twists along non-separating two-sided simple closed curves are elements of $\Gamma _2(N_g)$. Hence $\{a_1^2, \dots, a_{g-1}^2, b^2, y \}$ is a normal generating set for $\Gamma _2(N_g)$ in $\mathcal{M}(N_g)$.

We now explain about the Tietze transformations. Let $G$ be a group with presentation $G=\bigl< X \big| R\bigr>$, where $X$ is a subset of $G$ and $R$ is a set consisting of words of elements of $X$. Then $G$ is isomorphic to the quotient group $F/K$, where $F$ is the free group which is generated by $X$ and $K$ is the normal subgroup of $F$ which is normally generated by $R$. Then the following transformations among presentations do not change the isomorphism class of $G$.

\begin{align*}
\bigl< X \big| R\bigr> &\longleftrightarrow \bigl< X \big| R\cup \{ k\} \bigr> \hspace{3.54cm} \text{for}\ k\in K-R, \\
&\longleftrightarrow \bigl< X\cup \{ v\} \big| R\cup \{vw^{-1} \} \bigr>  \hspace{2cm} \text{for}\ w\in F-X. \hspace{1cm} 
\end{align*}
These transformations are called the {\it Tietze transformations}. In this paper, we use these transformations without any comment when we deform presentations (or relations).

\section{Proof of Theorem ~\ref{thm-main} for $g\geq 4$}\label{proof-thm}

By the definition of $\Gamma _2(N_g)$ and surjectivity of $\rho _2:\mathcal{M}(N_g)\rightarrow \operatorname{Aut}(H_1(N_g;\mathbb Z_2),\cdot )$, we have the following short exact sequence.

\begin{eqnarray}\label{eq-Aut}
1\longrightarrow \Gamma _2(N_g)\longrightarrow \mathcal{M}(N_g)\stackrel{\rho _2}{\longrightarrow} \operatorname{Aut}(H_1(N_g;\mathbb Z_2),\cdot )\longrightarrow 1.
\end{eqnarray}

We have the finite presentation for $\mathcal{M}(N_g)$ (Theorem~\ref{Stukow}) and the normal generating set $\{a_1^2, \dots, a_{g-1}^2, b^2, y \}$ for $\Gamma _2(N_g)$ (Theorem~\ref{Szepietowski}). We can get a presentation for $\operatorname{Aut}(H_1(N_g;\mathbb Z_2),\cdot )$ by adding $\{a_1^2, \dots, a_{g-1}^2, b^2, y \}$ to the relations of the presentation for $\mathcal{M}(N_g)$ in Theorem~\ref{Stukow}. 

The relations $a_1^2= \cdots= a_{g-1}^2= b^2=1$ are nothing but relations (1) in Theorem~\ref{thm-main} clearly. 
By Claim~\ref{second_claim} and relations (1), we have
\begin{eqnarray*}
(a_1a_2a_3a_4a_5)^6
&=&(a_1^2a_2a_3a_4a_5)^5\\
&=&(a_2a_3a_4a_5)^5\\
&=&(a_2^2a_3a_4a_5)^4\\
&\vdots&\\
&=&a_5^2\\
&=&1.
\end{eqnarray*}
Hence we obtain the relation $b_2=(a_1a_2a_3a_4a_5b)^5$ in relation (7)
from relation (A8) for $i=1$.
Relations (2), (3), (4), (5) and (9) in Theorem~\ref{thm-main} are obtained from relations (A1), (A2), (A3),
(A4) and (A9b) in Theorem~\ref{Stukow}, and relations (1).

Since $y=1$ in $\operatorname{Aut}(H_1(N_g;\mathbb Z_2),\cdot )$, relations (B1), (B3), (B4), (B7), (B8) are unnecessary. By using relations (1) and braid relations (relations (A2), (A4) in Theorem~\ref{Stukow}), relations (B2), (B5), (B6) are deformed as follows.

\begin{eqnarray*}
\text{(B2)}&\stackrel{y=1\& \text{(1)}}{\Longleftrightarrow }&a_2a_1a_2a_1a_2=\underline{a_1a_2a_1}a_2\underline{a_1a_2a_1}\stackrel{\text{BR}}{\Longleftrightarrow }a_2a_1a_2a_1a_2=a_2a_1\underline{a_2a_2}a_2a_1a_2  \\ 
&\stackrel{\text{(1)}}{\Longleftrightarrow }&a_2a_1a_2a_1a_2=a_2a_1a_2a_1a_2. \\
\text{(B5)}&\stackrel{y=1}{\Longleftrightarrow }&a_1=a_1^{-1}\Longleftrightarrow \text{(1)}. \\
\text{(B6)}&\stackrel{y=1\& \text{(1)}}{\Longleftrightarrow }&1=a_1\underline{a_2a_3a_2}a_3a_2a_1a_2a_3\underline{a_2a_3a_2}\stackrel{\text{BR}}{\Longleftrightarrow }1=a_1a_3\underline{a_2a_3a_3a_2}a_1\underline{a_2a_3a_3a_2}a_3 \\
&\Longleftrightarrow& 1=a_1a_3a_1a_3\stackrel{\text{(A1)}\& \text{(1)}}{\Longleftrightarrow }1=a_1^2.
\end{eqnarray*}

Therefore relations (B1), (B2), \dots , (B8) drop out. 

It is sufficient for proof of this theorem to show the following three claims:

\begin{claim}\label{rho}

Relations {\rm (C1a)}, {\rm (C4b)} are equivalent to $\rho =1$ under $y=1$, relations {\rm (1)}, {\rm (BR)} and {\rm (DI)}. It allows you to rule out generator $\rho $ and relations {\rm (C1a)}, {\rm (C2)}, {\rm (C3)} and {\rm (C4b)} from the presentation.

\end{claim}

\begin{claim}\label{second_claim}

Let $G$ be a group and assume that $g_1, g_2, \dots , g_n\in G$ satisfy relations
\begin{enumerate}
 \item[(BR)] $g_ig_{i+1}g_i=g_{i+1}g_ig_{i+1}$\hspace{1cm} for $i=1,\dots ,n-1 $,
 \item[(DI)] $[g_i,g_j]=1$\hspace{2.76cm} for $|i-j|>1$ .
\end{enumerate}
Then we have a relation $(g_1 g_2 \cdots g_n)^{n+1}=(g_1^2 g_2 \cdots g_n)^n$ on $G$.

``{\rm BR}'' and ``{\rm DI}'' means braid relations and disjointness relations, respectively.

\end{claim}

\begin{claim}\label{A9a}

Relation {\rm (A9a)} follows from relations {\rm(1), (2), (3), (4), (5)}.

\end{claim}

We suppose that Claim~\ref{second_claim} and Claim~\ref{A9a} are true.

\begin{proof}[\bf{Proof of Claim~\ref{rho}}]

Since $y=1$ in $\operatorname{Aut}(H_1(N_g;\mathbb Z_2),\cdot )$ and $\operatorname{Aut}(H_1(N_g;\mathbb Z_2),\cdot )$ has relations (C1a) and (C4b), $\rho $ is represented by the form
\[
\hspace{0.5cm}
\rho  =\left\{ \begin{array}{ll}
 (a_1a_2\cdots a_{g-1})^g  \hspace{1cm} &\text{for} \ g \ \text{odd}, \\
 (a_2\cdots a_{g-1})^{g-1} \hspace{1cm} &\text{for} \ g \ \text{even}. \\
 \end{array} \right.
\]

We get $\rho =1$ by repeatedly applying Claim~\ref{second_claim} and relations (1) to the right-hand side of the above equation. For example, in $g$ odd case:

\begin{align*}
\rho & = (a_1a_2\cdots a_{g-1})^g \stackrel{\text{Claim}~\ref{second_claim}}{=}(a_2\cdots a_{g-1})^{g-1} \stackrel{\text{Claim}~\ref{second_claim}}{=}(a_3\cdots a_{g-2})^{g-2} \\
&\stackrel{\text{Claim}~\ref{second_claim}}{=}\cdots \stackrel{\text{Claim}~\ref{second_claim}}{=}a_{g-1}^2\stackrel{(1)}{=}1.
\end{align*}
Thus we obtain the claim.
\end{proof}

By Claim~\ref{rho}, relations (C2) and (C3) are unnecessary. By a discussion similar to the proof of Claim~\ref{rho}, relations (C1b) and (C4a) are unnecessary, too. Therefore relations (C1a), (C1b), (C2), (C3), (C4a), (C4b) drop out.
For relation (A5), we apply Claim~\ref{second_claim} as follows.
\begin{align*}
\text{(A5)} &\Longleftrightarrow (a_2a_3a_4b)^{10}=\underline{(a_1a_2a_3a_4b)^6}\stackrel{\text{Claim}~\ref{second_claim}}{\Longleftrightarrow }(a_2a_3a_4b)^{10}=(a_2a_3a_4b)^5 \\
&\Longleftrightarrow (a_2a_3a_4b)^5=1 \stackrel{\text{Claim}~\ref{second_claim}}{\Longleftrightarrow }\cdots \stackrel{\text{Claim}~\ref{second_claim}}{\Longleftrightarrow }b^2=1 \Longleftrightarrow (1).
\end{align*}

We have completed the proof of Theorem~\ref{thm-main} without proofs of Claim~\ref{second_claim} and Claim~\ref{A9a}.

\begin{proof}[\bf{Proof of Claim~\ref{second_claim}}]

\[
(g_1 g_2 \cdots g_n)^{n+1}=(g_1 g_2 \cdots g_n)(g_1 g_2 \cdots g_n)\cdots (g_1 g_2 \cdots g_n)
\]

Let $A_i\ (i=n+1,\ n,\ \dots ,\ 1)$ be the $i$-th sequence $(g_1 g_2 \cdots g_n)$ from the right in the right-hand side. By using disjointness relations and braid relations, the above equation is deformed as follows. 
\begin{eqnarray*} 
(g_1 g_2 \cdots g_n)^{n+1}&=&A_{n+1}A_n\cdots A_1\\
&=&(g_1 g_2 \cdots g_{n-1}\underline{g_n})A_nA_{n-1}\cdots A_1 \\
&\stackrel{\text{DI}}{=}&(g_1 g_2 \cdots g_{n-1})(g_1 g_2 \cdots g_{n-2}\underline{g_ng_{n-1}g_n})A_{n-1}\cdots A_1 \\
&\stackrel{\text{BR}}{=}&(g_1 g_2 \cdots g_{n-1})(g_1 g_2 \cdots g_n)g_{n-1}A_{n-1}\cdots A_1.
\end{eqnarray*}
We replace the first sequence $(g_1 g_2 \cdots g_n)$ from the left in the bottom with $A_n$.
Then we have
\begin{eqnarray*} 
(g_1 g_2 \cdots g_n)^{n+1}&=&(g_1 g_2 \cdots g_{n-1})A_n\underline{g_{n-1}}A_{n-1}\cdots A_1 \\
&\stackrel{\text{DI}}{=}&(g_1 g_2 \cdots g_{n-1})A_n(g_1 g_2 \cdots g_{n-3}\underline{g_{n-1}g_{n-2}g_{n-1}}g_n)A_{n-2}\cdots A_1 \\
&\stackrel{\text{BR}}{=}&(g_1 g_2 \cdots g_{n-1})A_n(g_1 g_2 \cdots g_{n-3}g_{n-2}g_{n-1}\underline{g_{n-2}}g_n)A_{n-2}\cdots A_1 \\
&\stackrel{\text{DI}}{=}&(g_1 g_2 \cdots g_{n-1})A_n(g_1 g_2 \cdots g_n)g_{n-2}A_{n-2}\cdots A_1.
\end{eqnarray*}
We replace the second sequence $(g_1 g_2 \cdots g_n)$ from the left in the bottom with $A_{n-1}$ and repeat it.
Then we have
\begin{eqnarray*} 
(g_1 g_2 \cdots g_n)^{n+1}&=&(g_1 g_2 \cdots g_{n-1})A_nA_{n-1}g_{n-2}A_{n-2}\cdots A_1 \\
&=&(g_1 g_2 \cdots g_{n-1})A_nA_{n-1}A_{n-2}g_{n-3}A_{n-3}\cdots A_1 \\
&\vdots & \\
&=&(g_1 g_2 \cdots g_{n-1})A_n\cdots A_2g_1A_1 \\
&=&(g_1 g_2 \cdots g_{n-2})A_ng_{n-2}A_{n-1}\cdots A_2g_1A_1 \\
&\vdots &\\
&=&(g_1 g_2 \cdots g_{n-2})A_n\cdots A_3g_1A_2g_1A_1 \\
&\vdots &\\
&=&g_1A_n\cdots g_1A_3g_1A_2g_1A_1 \\
&=&(g_1^2 g_2 \cdots g_n)^n.
\end{eqnarray*}
Thus we obtain the claim.
\end{proof}

\begin{proof}[\bf{Proof of Claim~\ref{A9a}}]

Note that $b_2=(a_1a_2a_3a_4a_5b)^5$. 
We first show the followings.
\begin{enumerate}
 \item[(a)] $a_i(a_1a_2a_3a_4a_5b)=(a_1a_2a_3a_4a_5b)a_{i-1}$ \hspace{2cm}\ \text{for}\ $i=2,\ 3,\ 4$.
 \item[(b)] $b(a_1a_2a_3a_4a_5b)=(a_1a_2a_3a_4a_5b)a_4a_5a_4$.
 \item[(c)] $a_5(a_1a_2a_3a_4a_5b)=(a_1a_2a_3a_4a_5b)a_4ba_4$.
\end{enumerate}
Relation (a) is obtained by an argument similar to the proof of Claim~\ref{second_claim}. The other relations are obtained by the following deformations.
\begin{enumerate}
 \item[(b)] $b(a_1a_2a_3a_4a_5b)\stackrel{\text{DI}}{=}a_1a_2a_3\underline{ba_4b}a_5\stackrel{\text{BR}}{=}a_1a_2a_3a_4ba_4a_5\stackrel{(1)}{=}a_1a_2a_3a_4b(a_5\underline{a_5)a_4a_5}$ \\
$\stackrel{\text{BR}}{=}(a_1a_2a_3a_4a_5b)a_4a_5a_4$.
 \item[(c)] $a_5(a_1a_2a_3a_4a_5b)\stackrel{\text{DI}}{=}a_1a_2a_3\underline{a_5a_4a_5}b\stackrel{\text{BR}}{=}a_1a_2a_3a_4a_5a_4b\stackrel{(1)}{=}a_1a_2a_3a_4a_5(b\underline{b)a_4b}$ \\
$\stackrel{\text{BR}}{=}(a_1a_2a_3a_4a_5b)a_4ba_4$.
\end{enumerate}
%

We now prove $bb_2=b_2b$ by using only relations (a), (b), (c), 
(1) and disjointness relations. It means the relation $bb_2=b_2b$ is unnecessary.
\begin{eqnarray*}
bb_2&=&b(a_1a_2a_3a_4a_5b)^5\\
&\stackrel{\text{(b)}}{=}&(a_1a_2a_3a_4a_5b)a_4a_5a_4(a_1a_2a_3a_4a_5b)^4 \\
&\stackrel{\text{(a),(c)}}{=}&(a_1a_2a_3a_4a_5b)^2a_3a_4ba_4a_3(a_1a_2a_3a_4a_5b)^3 \\
&\stackrel{\text{(a),(b)}}{=}&(a_1a_2a_3a_4a_5b)^3a_2a_3a_4a_5a_4a_3a_2(a_1a_2a_3a_4a_5b)^2 \\
&\stackrel{\text{(a),(c)}}{=}&(a_1a_2a_3a_4a_5b)^4a_1a_2a_3a_4b\underline{a_4a_3a_2a_1(a_1a_2a_3a_4}a_5b) \\
&\stackrel{(1)}{=}&(a_1a_2a_3a_4a_5b)^4a_1a_2a_3a_4\underline{ba_5}b \\
&\stackrel{\text{DI}}{=}&(a_1a_2a_3a_4a_5b)^5b\\
&=&b_2b.
\end{eqnarray*}
Thus we obtain the claim.
\end{proof}

\section{The second homology group of $\operatorname{Aut}(H_1(N_g;\mathbb Z_2),\cdot )$}\label{H_1andH_2}

In this section, we prove Theorem~\ref{h_2}.
First, we obtain a generating set for $H_2(\operatorname{Aut}(H_1(N_g;\mathbb Z_2),\cdot );\mathbb Z)$ when $g\geq 9$ by using the Hopf formula and applying the discussion of Pitsch \cite{Pitsch}. More precisely, we obtain the following proposition.

\begin{prop}\label{generator}
For $g\geq 9$, $H_2(\operatorname{Aut}(H_1(N_g;\mathbb Z_2),\cdot
 );\mathbb Z)$ is generated by one element $x_0$. $x_0$ is represented by the following element:
\[
A^{-7}B_1^{-2}B_2^{-4}B_3^{-6}B_4^{4}B_5^{2}B^{12}C^2, 
\]
where $A$, $B_i$ $(i=1,\dots,5)$, $B$ and $C$ are the followings.
\begin{eqnarray*}
A&:=&b^2,\\
B_i&:=&a_ia_{i+1}a_ia_{i+1}^{-1}a_i^{-1}a_{i+1}^{-1},\\
B&:=&ba_4ba_4^{-1}b^{-1}a_4^{-1},\\
C&:=&(a_2a_3a_4a_5a_6b)^6(a_2^{-1}a_3^{-1}a_4^{-1}a_5^{-1}a_6^{-1}b^{-1})^6(a_1a_2a_3a_4a_5a_6b)^{-4}\\
&&\cdot(a_1^{-1}a_2^{-1}a_3^{-1}a_4^{-1}a_5^{-1}a_6^{-1}b^{-1})^{-5}.
\end{eqnarray*}
\end{prop}

Now we recall the classical Hopf formula. Let $G$ be a group with finite presentation $G=\bigl< X \big| R\bigr>$, where $X$ is a finite subset of $G$ and $R$ is a finite set consisting of words of the elements of $X$. Then $G$ is isomorphic to the quotient group $F/K$, where $F$ is the free group which is generated by $X$ and $K$ is the normal subgroup of $F$ which is normally generated by $R$. The classical {\it Hopf formula} states that
\[
H_2(G;\mathbb Z)\cong \frac {K\cap [F,F]}{[K,F]}  .
\]
We remark that $(K\cap [F,F])/[K,F]$ is an abelian group and any element of $(K\cap [F,F])/[K,F]$ is represented by a product of commutators of elements of $F$ and by a product of conjugations of elements of $R$ on $F$. Since $fkf^{-1}\equiv k$ in $K/[K,F]$ for any $f\in F$ and $k\in K$, every element of $H_2(G;\mathbb Z)$ is represented by $\Pi \ r_i^{n_i}$, where $R=\{ r_1, \dots , r_N\}$ and $n_i\in \mathbb Z$.

We modify the presentation for $\operatorname{Aut}(H_1(N_g;\mathbb Z_2),\cdot )$ for $g\geq 9$ in Theorem~\ref{thm-main} to apply the Hopf formula to $\operatorname{Aut}(H_1(N_g;\mathbb Z_2),\cdot )$ easily. 

At first we easily know that $\operatorname{Aut}(H_1(N_g;\mathbb Z_2),\cdot )$ admits a presentation with generators $a_0, a_1, \dots , a_{g-1}$ and relators:

\begin{enumerate}
 \item[(1)] $a_i^2$\hspace{1.0cm} for $i=0,\dots ,g-1$,
 \item[(2)] $[a_i,a_j]$\hspace{1.0cm} for ``$j-i>1$ and $i\not=0$'' or ``$i=0$ and $j\not=4$'', 
 \item[(3)] $a_ia_{i+1}a_ia_{i+1}^{-1}a_i^{-1}a_{i+1}^{-1}$\hspace{1.0cm} for $ i=1,\dots ,g-2$ \\
$a_0a_4a_0a_4^{-1}a_0^{-1}a_4^{-1}$, 
 \item[(4)] $(a_2a_3a_4a_5a_6a_0)^6(a_2^{-1}a_3^{-1}a_4^{-1}a_5^{-1}a_6^{-1}a_0^{-1})^6(a_1a_2a_3a_4a_5a_6a_0)^{-4}(a_1^{-1}a_2^{-1}a_3^{-1}\\ a_4^{-1}a_5^{-1}a_6^{-1}a_0^{-1})^{-5}$,
 \item[(5)] $[a_{g-5},b_{\frac{g-2}{2}}]$\hspace{1.0cm} for $g\geq 8$ even,
\end{enumerate}
where $a_0=b$ and $b_{\frac{g-2}{2}}$ is inductively defined as follows:
$b_1=a_0$, $b_2=(a_1a_2a_3a_4a_5b)^5$, 
$b_{i+1}=(b_{i-1}a_{2i}a_{2i+1}a_{2i+2}a_{2i+3}b_i)^5(b_{i-1}a_{2i}a_{2i+1}a_{2i+2}a_{2i+3})^{-6}$ for $2\leq i\leq \frac{g-4}{2}$. 

\begin{lem}

In the above presentation, relators $a_1^2,\dots , a_{g-1}^2$ in {\rm(1)} are unnecessary.

\end{lem}
   
\begin{proof}

By the relators (3), we can write $a_1, \dots , a_{g-1}$ as conjugations of $a_0$ in $\operatorname{Aut}(H_1(N_g;\mathbb Z_2),\cdot )$ inductively, as follows.
\begin{align*}
&\hspace{3cm}a_4=a_0a_4a_0a_4^{-1}a_0^{-1}, \\
&a_5=a_4a_5a_4a_5^{-1}a_4^{-1}, \hspace{3cm} a_3=a_4a_3a_4a_3^{-1}a_4^{-1}, \\
&a_6=a_5a_6a_5a_6^{-1}a_5^{-1}, \hspace{3cm} a_2=a_3a_2a_3a_2^{-1}a_3^{-1}, \\
&\hspace{0.52cm} \vdots  \hspace{5.52cm} a_1=a_2a_1a_2a_1^{-1}a_2^{-1}. \\
&a_{g-1}=a_{g-2}a_{g-1}a_{g-2}a_{g-1}^{-1}a_{g-2}^{-1}, \hspace{4.5cm} 
\end{align*} 

Thus $a_1^2,\ \cdots ,\ a_{g-1}^2$ are cojugations of $a_0^2$ in $\operatorname{Aut}(H_1(N_g;\mathbb Z_2),\cdot )$.

\end{proof}

We set
\begin{eqnarray*}
A&:=&a_0^2,\\
D_{i,j}&:=&[a_i,a_j],\\
B_i&:=&a_ia_{i+1}a_ia_{i+1}^{-1}a_i^{-1}a_{i+1}^{-1},\\
B&:=&a_0a_4a_0a_4^{-1}a_0^{-1} a_4^{-1},\\
C&:=&(a_2a_3a_4a_5a_6a_0)^6(a_2^{-1}a_3^{-1}a_4^{-1}a_5^{-1}a_6^{-1}a_0^{-1})^6(a_1a_2a_3a_4a_5a_6a_0)^{-4}\\
&&(a_1^{-1}a_2^{-1}a_3^{-1}a_4^{-1}a_5^{-1}a_6^{-1}a_0^{-1})^{-5},\\
D&:=&[a_{g-5},b_{\frac{g-2}{2}}].
\end{eqnarray*}
Then any element $x$ of $H_2(\operatorname{Aut}(H_1(N_g;\mathbb Z_2),\cdot );\mathbb Z)$ is represented by 
\[
x=A^n\Biggl( \prod _{(\ast )} D_{i,j}^{n_{i,j}}\Biggr) \Biggl( \prod _{i=1}^{g-2}B_i^{m_i}\Biggr) B^mC^lD^{l^{\prime }},
\]
where $n,\ n_{i,j},\ m_i,\ m,\ l,\ l^{\prime }\in \mathbb Z$ and
$(\ast )$ means the condition ``$j-i>1$ and $i\not=0$'' or ``$i=0$ and $j\not=4$''.

\begin{defn}\label{defcommutator}
Let $G$ and $F$ be groups which are given in the Hopf formula. For $g,h\in F$ such that $[g,h]=1$ in $G$, we denote by $\{g,h\}$ the equivalence class of the commutator $[g,h]\in [F,F]$ in $H_2(G;\mathbb Z)$. 
\end{defn}
Korkmaz and Stipsicz \cite[Lemma 3.3]{Korkmaz-Stipsicz} give the following relations in $H_2(G;\mathbb Z)$. For $g, h, k\in G$ such that $g$ commute with $h$ and $k$,
\begin{align*}
&(I)\ \ \{ g, hk\}=\{ g, h\} +\{ g, k\}, \\
&(I\hspace{-.25em}I )\ \ \{ g, h^{-1}\} =-\{ g,h\}.
\end{align*}
Note that relation (I) is obtained from relation (I\hspace{-.1em}I).

Let $\mathcal T(N_{g,n})$ be the subgroup of $\mathcal M(N_{g,n})$ generated by all Dehn twists and $\mathcal M(\Sigma _{g,n})$ the mapping class group of a compact connected orientable surface $\Sigma _{g,n}$ of genus $g$ with $n$ boundary components (i.e. $\mathcal M(\Sigma _{g,n})$ is the group of isotopy classes of orientation preserving self-diffeomorphisms on $\Sigma _{g,n}$ which fix each boundary component pointwise).  
 
\begin{lem} \label{commutator}

Let $g\geq 9$. If $\alpha $ and $\beta $ are disjoint non-separating two-sided simple closed curves on $N_g$ then $\{ t_\alpha ,t_\beta \}=0$ in $H_2(\mathcal T(N_g);\mathbb Z)$, where $t_\alpha $ and $t_\beta $ are Dehn twists along simple closed curves $\alpha $ and $\beta $ respectively. 

\end{lem}

\begin{proof}

Let $S$ be the surface obtained by cutting $N_g$ along the simple closed curve $\alpha $ and $g^{\prime }$ the genus of $S$. Note that if $g$ is even and $S$ is orientable then $g^\prime = \frac {g-2}{2}\geq \frac {10-2}{2}=4$ and if $g$ is odd or $S$ is non-orientable then $g^{\prime }=g-2\geq 7$ since $g\geq 9$. We regard $t_\beta $ as an element of $\mathcal M(\Sigma _{g^\prime ,2})$ when $g$ is even and $S$ is orientable or $\mathcal T(N_{g^\prime ,2})$ when $g$ is odd or $S$ is non-orientable. Harer \cite{Harer} proved that $H_1(\mathcal M(\Sigma _{h,n});\mathbb Z)=1$ for $h\geq 3$ and Stukow \cite{Stukow2} proved that $H_1(\mathcal T(N_{h,n});\mathbb Z)=1$ for $h\geq 7$. Thus there exist $X_i,\ Y_i\in \mathcal M(S)$ or $\mathcal T(S)$ such that $t_\beta =\prod _i [X_i,\ Y_i]$. Note that $X_i$ and $Y_i$ commute with $t_\alpha $. Therefore, in $H_2(\mathcal{T}(N_g);\mathbb{Z})$, we have
\begin{eqnarray*}
\{ t_\alpha , t_\beta \} &=&\Bigl\{ t_\alpha , \prod _i [X_i, Y_i]\Bigr\} \stackrel{(I)}{=}\sum _i\{ t_\alpha , [X_i, Y_i]\} \\
&\stackrel{(I)\& (I\hspace{-.25em}I )}{=}&\sum _i\Bigl[\{ t_\alpha , X_i\} +\{ t_\alpha , Y_i \} -\{ t_\alpha , X_i\} -\{ t_\alpha , Y_i \} \Bigr] \\
&=&0.
\end{eqnarray*}
Thus we obtain the claim.
\end{proof}

The homomorphism $\rho _2|_{\mathcal T(N_g)}:\mathcal T(N_g)\rightarrow \operatorname{Aut}(H_1(N_g;\mathbb Z_2),\cdot )$ induces a homomorphism $H_2(\mathcal T(N_g);\mathbb Z)\rightarrow H_2(\operatorname{Aut}(H_1(N_g;\mathbb Z_2),\cdot );\mathbb Z)$ on their homology groups. Hence the equivalence classes of $D_{i,j}$ and $D$ in $H_2(\operatorname{Aut}(H_1(N_g;\mathbb Z_2),\cdot );\mathbb Z)$ are trivial by Lemma~\ref{commutator} and any element of $H_2(\operatorname{Aut}(H_1(N_g;\mathbb Z_2),\cdot );\mathbb Z)$ is represented by
\[
x=A^n \Biggl( \prod _{i=1}^{g-2}B_i^{m_i}\Biggr) B^mC^l.
\] 

\begin{proof}[\bf{Proof of Proposition~\ref{generator}}]

By the Hopf formula, any element of $H_2(\operatorname{Aut}(H_1(N_g;\mathbb Z_2),\cdot );\mathbb Z)$ is a product of commutators of the free group generated by $\{a_0, a_1, \cdots , a_{g-1}\}$. Hence the exponent sum of each $a_i$ in $x$ is zero. The exponent sum of each $a_i$ in $x$ is the following.
\begin{align*}
(\text{the exponent sum of }a_0)&=2n+m+l, \\
(\text{the exponent sum of }a_1)&=m_1+l, \\
(\text{the exponent sum of }a_2)&=-m_1+m_2+l, \\
(\text{the exponent sum of }a_3)&=-m_2+m_3+l, \\
(\text{the exponent sum of }a_4)&=-m_3+m_4-m+l, \\
(\text{the exponent sum of }a_5)&=-m_4+m_5+l, \\ 
(\text{the exponent sum of }a_6)&=-m_5+m_6+l, \\
(\text{the exponent sum of }a_7)&=-m_6+m_7, \\
&\ \hspace{0.07cm} \vdots \\
(\text{the exponent sum of }a_{g-2})&=-m_{g-3}+m_{g-2}, \\
(\text{the exponent sum of }a_{g-1})&=-m_{g-2}.
\end{align*}
The above equations give $m_{g-2}=m_{g-3}=\cdots =m_7=m_6=0$ and the following system of equations.
\[
\left(
\begin{array}{ccccccccccccccccccc}
2 & 0 & 0 & 0 & 0 & 0 & 1  & 1   \\
0 & 1 & 0 & 0 & 0 & 0 & 0  & 1   \\
0 &-1 & 1 & 0 & 0 & 0 & 0  & 1   \\
0 & 0 &-1 & 1 & 0 & 0 & 0  & 1   \\
0 & 0 & 0 &-1 & 1 & 0 & -1 & 1   \\
0 & 0 & 0 & 0 &-1 & 1 & 0  & 1   \\
0 & 0 & 0 & 0 & 0 &-1 & 0  & 1   \\
\end{array}
\right)
\left(
\begin{array}{ccccccccccccccccccc}
n \\
m_1 \\
m_2 \\
m_3 \\
m_4 \\
m_5 \\
m \\
l
\end{array}
\right)
=
\left(
\begin{array}{ccccccccccccccccccc}
0 \\
0 \\
0 \\
0 \\
0 \\
0 \\
0 
\end{array}
\right).
\]
By an elementary calculation, this matrix has rank $7$ and so the linear map $\mathbb Z^8\rightarrow \mathbb Z^7$ has a 1-dimensional kernel. We can check the kernel is generated by the vector $(-7,-2,-4,-6,4,2,12,2)$. Therefore $H_2(\operatorname{Aut}(H_1(N_g;\mathbb Z_2),\cdot );\mathbb Z)$ is generated by $x_0$ which is represented by an element
\[
A^{-7}B_1^{-2}B_2^{-4}B_3^{-6}B_4^{4}B_5^{2}B^{12}C^2.
\]
Thus we finish the proof.
\end{proof}

When $g\geq 7$ is odd, Theorem~\ref{h_2} is proved by Stein \cite{Stein}. It is sufficient for a proof of Theorem~\ref{h_2} to show that $x_0=0$ when $g\geq 10$ is even.

\begin{proof}[\bf{Proof of Theorem~\ref{h_2}}]

Recall that $\operatorname{Aut}(H_1(N_g;\mathbb Z_2),\cdot )$ is isomorphic to $O(g,\mathbb Z_2)=\{A\in GL(g,\mathbb Z_2)\ |\ ^t\!AA=E\}$. Under this identification, we define the injective homomorphism

\[
\begin{array}{cccc}
\iota _g: & \operatorname{Aut}(H_1(N_{g-1};\mathbb Z_2),\cdot )                     & \hookrightarrow  & \operatorname{Aut}(H_1(N_g;\mathbb Z_2),\cdot )                     \\
& \rotatebox{90}{$\in$} &                 & \rotatebox{90}{$\in$} \\
& \mbox{\smash{\huge $A$}}                    & \mapsto     & \left(
\begin{array}{ccc|c}
&&&0\\
&\mbox{\smash{\huge $A$}}&&\vdots\\
&&&0\\ \hline
0&\cdots &0&1
\end{array}
\right).

\end{array}
\]

Note that $\iota _g(a_i)=a_i$ for $i=1,\dots ,g-2$ and $\iota _g(b)=b$. Let $F$ and $F^\prime $ be free groups generated by $\{a_1,\dots ,a_{g-1},b\}$ and $\{a_1,\dots ,a_{g-2},b\}$ respectively and $\nu :F\rightarrow \operatorname{Aut}(H_1(N_g;\mathbb Z_2),\cdot ) $ and $\nu ^\prime :F^\prime \rightarrow \operatorname{Aut}(H_1(N_{g-1};\mathbb Z_2),\cdot ) $ natural projections. Then the following diagram is commutative.

\[
\xymatrix{
F^\prime  \ar[d]_{\nu ^\prime } \ar[r]^{\widetilde{\iota _g}} \ar@{}[dr]|\circlearrowleft & F \ar[d]^{\nu } \\
\operatorname{Aut}(H_1(N_{g-1};\mathbb Z_2),\cdot ) \ar[r]^{\iota _g} & \operatorname{Aut}(H_1(N_g;\mathbb Z_2),\cdot ) \\
}\
\]

The homomorphism $\widetilde{\iota _g}:F^\prime \rightarrow F$ is defined by $\widetilde{\iota _g}(a_i)=a_i$ for $i=1,\dots ,g-2$ and $\widetilde{\iota _g}(b)=b$. We denote the kernels of $\nu $ and $\nu ^\prime $ by $K$ and $K^\prime $ respectively. By the Hopf formula, the restriction $\widetilde{\iota _g}:K^\prime \cap [F^\prime ,F^\prime ]\rightarrow K\cap [F,F]$ of $\widetilde{\iota _g}$ induces the homomorphism $\widetilde{\iota _g}_\ast : H_2(\operatorname{Aut}(H_1(N_{g-1};\mathbb Z_2),\cdot );\mathbb Z) \rightarrow H_2(\operatorname{Aut}(H_1(N_g;\mathbb Z_2),\cdot );\mathbb Z)$. Since $H_2(\operatorname{Aut}(H_1(N_{g-1};\mathbb Z_2),\cdot );\mathbb Z)=0$ for $g\geq 10$ even(\cite{Stein}), it is enough for the proof of Theorem~\ref{h_2} to show that $\widetilde{\iota _g}_\ast $ is surjective for $g\geq 10$. By Proposition~\ref{generator}, $H_2(\operatorname{Aut}(H_1(N_g;\mathbb Z_2),\cdot );\mathbb Z)$ is generated by $x_0$ for $g\geq 9$ such that $x_0$ is represented by $A^{-7}B_1^{-2}B_2^{-4}B_3^{-6}B_4^{4}B_5^{2}B^{12}C^2$. Thus we can check $\widetilde{\iota _g}_\ast (x_0^\prime )=x_0$ by the definition of $\widetilde{\iota _g}$, where $x_0^\prime $ is represented by an element $A^{-7}B_1^{-2}B_2^{-4}B_3^{-6}B_4^{4}B_5^{2}B^{12}C^2 \in K^\prime \cap [F^\prime ,F^\prime ]$. Therefore $x_0$ is trivial and we complete the proof.

\end{proof}
\par
{\bf Acknowledgement: } The authors would like to express his gratitude to Hisaaki Endo, for his encouragement and helpful advices. The authors also wish to thank Susumu Hirose for his comments and helpful advices.

\end{document}